\documentclass[10pt]{amsart}
\usepackage{amsmath}
\usepackage{amssymb} 
\usepackage{epsbox}
\usepackage{amsthm}
\usepackage{graphicx}

\begin{document}  
\title{Braid ordering and knot genus}
\author{Tetsuya Ito}
\address{Graduate School of Mathematical Science, University of Tokyo, 3-8-1 Komaba, Meguro-ku, Tokyo, 153-8914, Japan}
\email{tetitoh@ms.u-tokyo.ac.jp}

 
\newtheorem{thm}{Theorem}
\newtheorem{cor}{Corollary}
\newtheorem{lem}{Lemma}
\newtheorem{prop}{Proposition}
 {\theoremstyle{definition}
  \newtheorem{defn}{Definition}}
 \newcommand{\C}{\mathbb{C}}


\begin{abstract}
   The genus of knots is one of the fundamental invariant and can be seen as a complexity of knots.
In this paper we give a lower bound of the knot genus using the Dehornoy floor, which is a measure of complexity of braids in terms of the Dehornoy ordering. 
\end{abstract}
\maketitle

\section{Introduction}

   Let $B_{n}$ be the degree $n$ braid group, defined by the presentation
\[
B_{n} = 
\left\langle
\sigma_{1},\sigma_{2},\cdots ,\sigma_{n-1}
\left|
\begin{array}{ll}
\sigma_{i}\sigma_{j}=\sigma_{j}\sigma_{i} & |i-j|\geq 2 \\
\sigma_{i}\sigma_{j}\sigma_{i}=\sigma_{j}\sigma_{i}\sigma_{j} & |i-j|=1 \\
\end{array}
\right.
\right\rangle.
\]
 
  The {\it Dehornoy ordering} $<_{D}$ is a left-invariant total ordering of the braid group $B_{n}$ defined as follows. For two $n$-braids $\alpha$ and $\beta$, we define $\alpha <_{D} \beta$ if and only if the braid $\alpha^{-1}\beta$ admit a word representative which contains at least one $\sigma_{i}$ and contains no $\sigma_{1}^{\pm1},\sigma_{2}^{\pm1}\cdots ,\sigma_{i-1}^{\pm 1},\sigma_{i}^{-1}$ for some $1\leq i \leq n-1$. Many other equivalent definitions of the Dehornoy ordering are known, hence the Dehornoy ordering is a quite natural structure of the braid group $B_{n}$ \cite{ddrw1},\cite{ddrw2}.

Let $\Delta =  (\sigma_{1}\sigma_{2}\cdots\sigma_{n-1})(\sigma_{1}\sigma_{2}\cdots\sigma_{n-2})\cdots(\sigma_{1}\sigma_{2})(\sigma_{1})\in B_{n}$ be the Garside fundamental braid. The Garside fundamental braid has special properties and plays an important role in the braid group. For example, the center of the braid group is an infinite cyclic group which is generated by $\Delta^{2}$.
Using the Dehornoy ordering and the Garside fundamental braid $\Delta$, for each braid $\beta$, we define the {\it Dehornoy floor} $[\beta]_{D}$, which is a measure of a complexity of braids, as follows.

\begin{defn}[Dehornoy floor]
  The Dehornoy floor $[\beta]_{D}$ of a braid $\beta \in B_{n}$ is a non-negative integer defined by  
\[ [\beta]_{D} = \min \{ m \in \mathbb{Z}_{\geq 0}\: |  \Delta^{-2m-2} <_{D} \beta  <_{D} \Delta^{2m+2}\:  \}. \]
\end{defn}

  The purpose of this paper is to compare the Dehornoy floor, which is a fundamental complexity of a braid in the Dehornoy ordering view point, and the genus of links described as a closure of the braid, which is the most fundamental complexity of links in topological view point. 

   Our main result is the following.
\begin{thm}
\label{thm:main}
Let $\beta \in B_{n}$ be a braid and $\chi(\widehat{\beta})$ be the maximal Euler characteristics of an orientable spanning surface whose boundary is $\widehat{\beta}$. Then, the inequality 
\[ [\beta ]_{D} < \frac{3}{2}-\frac{2\chi(\widehat{\beta})}{n+2}\]
holds.
\end{thm}

   As a consequence, we obtain a new lower bound of knot genus.
\begin{cor}
\label{c1}
Let $K$ be an oriented knot and $g(K)$ be the genus of $K$. If $K$ is represented as the closure of an $n$-braid $\beta$, then the inequality 
\[ [\beta ]_{D} < \frac{4g(K)}{n+2} -\frac{2}{n+2}+ \frac{3}{2} \leq g(K)+1\]
holds. 
\end{cor}

  Thus we conclude that the closure of a complex braid with respect to the Dehornoy ordering is also complex with respect to a topological view point.\\

\textbf{Acknowledgments.}
 The author gratefully acknowledges the many helpful suggestions of professor Toshitake Kohno during the preparation of the paper. This research was supported by JSPS Research Fellowships for Young Scientists.

\section{Preliminaries}

   In this section, we present some of basic facts of the Dehornoy ordering and Birman-Menasco's braid foliation theory which will be used in later.

\subsection{Properties of the Dehornoy floor}
  
  In this subsection we review some properties of the Dehornoy ordering and the Dehornoy floor. For details, see \cite{i}.
  For $1\leq i < j \leq n$, let $\{a_{i,j}\}$ be an $n$-braid defined by 
\[ a_{i,j}= (\sigma_{i}\sigma_{i+1}\cdots \sigma_{j-2})\sigma_{j-1}(\sigma_{j-2}^{-1}\cdots\sigma_{i+1}^{-1} \sigma_{i}^{-1}). \]

 The braids $\{a_{i,j}\}_{1<i<j\leq n}$ are called the {\it band generators}. 
 The following properties of the Dehornoy floor are proved by Malyutin, Netsvetaev and the author.
 
\begin{prop}[\cite{mn},\cite{ma},\cite{i}]
\label{prop:dehornoyfloor}
   Let $\alpha,\beta \in B_{n}$. Then the following holds.
\begin{enumerate}
\item If a braid $\beta$ is conjugate to a braid which is represented by a word which contains $s$ $\sigma_{1}$ and $k$ $\sigma_{1}^{-1}$, then $[\beta]_{D} < max\{s,k\}$ holds.
\item $|[\beta]_{D}-[\alpha]_{D}|\leq 1$ if $\alpha$ and $\beta$ are conjugate.
\item $[\alpha \beta]_{D} \leq [\alpha]_{D} + [\beta]_{D} + 1$.
\item If a braid $\beta$ is conjugate to a braid represented as a product of $m$ band generators, then $[\beta]_{D} < \frac{m}{n}$ holds.
\end{enumerate}
\end{prop}

These properties will be used later to estimate the Dehornoy floor.

\subsection{Braid foliation theory}
In this subsection, we summarize a basic machinery of Birman-Menasco's braid foliation theory in the case of the incompressible Seifert surface. For details of the braid foliation theories, see \cite{bf}.

   Fix an unknot $A \in S^{3}$, called {\it axis} and choose a meridinal disc fibration $H = \{ H_{\theta}\;|\; \theta \in [0,2\pi]\}$ of the solid torus $S^{3} \backslash A$.
An oriented link $L$ in $S^{3} \backslash A$ is called a {\it closed braid} with axis $A$ if $L$ intersects every fiber $H_{\theta}$ transversely and each fiber is oriented so that all intersections of $L$ are positive. 
  
   Let $F$ be an oriented, connected spanning surface of $L$ with the maximal Euler characteristics. An orientation of $F$ is defined so that $L = \partial F$ holds. We remark that such a surface is always incompressible in $S^{3} \backslash L$. Then the intersections of fiber $\{H_{\theta}\}$ with $F$ induce a singular foliation of $F$, whose leaves consist of the connected components of the intersection with fibers. The braid foliation techniques are, in short, modifying this foliation simpler as possible and obtain a standard position or representation of braids and surfaces. 
 We say a fiber $H_{\theta}$ is {\it regular} if $H_{\theta}$ transverse $F$ and {\it singular}
if $H_{\theta}$ tangent to $F$. 
A non-singular leaf is called {\it a-arc} if one of its boundary point lies on $L$ and the other lies on $A$. Similarly, non-singular leaf is called a {\it b-arc} if both of its boundary points lie on $A$. 
We say b-arc $b$ is {\it essential} if the both connected components of $H_{\theta} \backslash b$ are pierced by $L$. 
   
$F$ can be isotoped to a good position with respect to the fibration which satisfies the following conditions, which we call {\it braid-foliation conditions} \cite{bf}.
\begin{enumerate}
\item Axis $A$ pierces $F$ transversely in finitely many points.
\item For each point $v \in A\cap F$, there exists a neighborhood $N_{v}$ of $v$ such that $F\cap N_{v}$ is a radially foliated disc.
\item All but finitely many fibers $H_{\theta}$ intersect $F$ transversely, and each of the exceptional fiber is tangent to $F$ at exactly one point. Moreover, each point of tangency is a saddle tangency and lies in the interior of $F \cap H_{\theta}$.
\item All regular leaves are a-arc or b-arc, and every b-arc is essential.
\end{enumerate}

   Now assume that $F$ satisfies the braid foliation conditions. We call an intersection point of $A\cap F$ {\it vertex}. For each vertex $p$, the valance of vertex $p$ is, by definition, the number of singular leaves which pass $p$. 
We say a singular point is an {\it aa-singularity} if the singular point is derived from two a-arcs.
An {\it ab-singularity} and a {\it bb-singularity} are defined by the same way. Each type of singularity has a foliated neighborhood as shown in the figure \ref{fig:regions}. We call these neighborhoods of singularities {\it regions}. The decomposition of the surface $F$ into regions defines a cellular decomposition of $F$.
It is directly checked that the valance of a vertex which is previously defined, the number of singular leaves which pass the vertex, coincide with the usual definition of the valance in this cellular decomposition.

\begin{figure}[htbp]
 \begin{center}
\includegraphics[width=70mm]{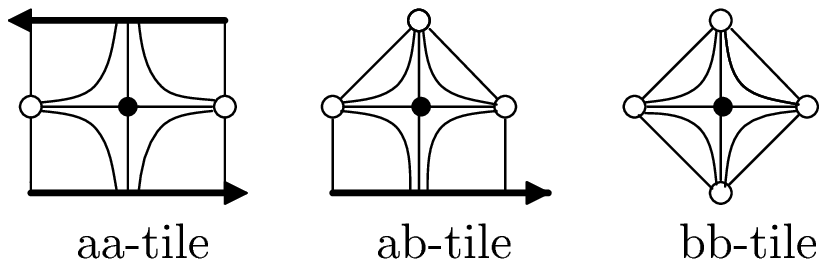}
 \end{center}
 \caption{aa-,ab-,bb- tiles}
 \label{fig:regions}
\end{figure}
 
  We say the sign of a singular point $p$ in the fiber $H_{\theta}$ is {\it positive} (resp. {\it negative}) if the outward pointing normal vector of $F$ at $p$ agrees (resp. disagrees) with the normal vector of the fiber $H_{\theta}$ at $p$. 

The notion of sign is used to decrease the valance of a vertex by an isotopy of the surface $F$. 

\begin{lem}[Change of foliation (See \cite{bf})]
\label{lem:folchange}
Let $F$ be foliated surface as the above.
If there exist adjacent bb-singular points along the neighborhood of a vertex $v$ with the same signs, then by an isotopy of $F$, we can decrease the valance of $v$ by one.
\end{lem} 

\section{Proof of theorem \ref{thm:main}}

   In this section we prove theorem \ref{thm:main}. The strategy of proof is the following. We first establish the Euler characteristic formula for a Seifert surface, which relate the valance of vertices and the Euler characteristics of the surface. Next we estimate the Dehornoy floor of the braid using the valance of vertices. Combining these two results, we obtain the desired estimation.

  We call a $n$-punctured disc $D_{n} = \{ z \in \C \: |\: |z| \leq n+1 \} \backslash \{1,2,3,\cdots, n\}$ the {\it standard} $n$-{\it punctured disc}. We regard an $n$-braid $\beta$ as a move of puncture points of the standard $n$-punctured disc. We call the region $D_{n} \cap \{ z \in D_{n} \:|\: \textrm{Im} z \leq 0\}$ the {\it lower half of the standard punctured disc}.

   Let $F$ be a spanning surface of a closed $n$-braid $L=\widehat{\beta}$ with the maximal Euler characteristics. We isotope $F$ so that $F$ satisfies the braid-foliation conditions. 
 Let $V(a,b)$ be the number of vertices in the tiling of $F$ whose valance is $a+b$ and which have $a$ a-arcs as their edges and $b$ b-arcs as their edges. We call such a vertex {\it type} $(a,b)$-{\it vertex}.

 The following lemma has proved in \cite{bm} using the simple Euler characteristic argument.
   
\begin{lem}[Birman-Menasco(\cite{bm})]
\label{lem:eulerform}

\begin{eqnarray*}
     & \lefteqn{2V(1,0)+2V(0,2)+V(0,3)-4 \chi(F) } \hspace{1cm}  \\
  &=& V(2,1)+2V(3,0) + \sum_{v=4}^{\infty}\sum_{\alpha=0}^{v}(v+\alpha -4)V(v,\alpha ) \\
\end{eqnarray*}

\end{lem}
 
 Now we extract information of the closed braid $\widehat{\beta}$ via the braid foliation of $F$.
The following lemma is the core of this paper. 
 
 \begin{lem}
\label{lem:estimate}
 If $F$ has a type $(a,b)$-vertex, then the inequality
\[ [\beta]_{D} < a + \frac{b}{2} - \frac{1}{2}\]
 holds.
\end{lem} 

\begin{proof}

We prove the lemma by estimating the number of $\sigma_{1}^{\pm 1}$ in the braid $\beta$.
The proof of lemma is divided into five steps. In the first step, we modify the closed braid $\widehat{\beta}$ and the surface $F$ so that we can decompose the braid into smaller pieces. In the second step we obtain an explicit description of the modified braid by using braid foliation.
In the third step we describe the method to simplify the obtained description of the braid so that it contains less $\sigma_{1}^{\pm 1}$. The fourth step is devoted ti the proof in case of $a=0$ or $b=0$.
In the last step, we treat the case both $a$ and $b$ are non-zero.\\

\textbf{Step 1: Modifying surface and closed braid}\\

Let $v$ be a type $(a,b)$ vertex of $F$. Let $\{ H_{\theta_{i}} \; | \; i=1,2, \cdots, a + b, \; \; \theta_{i}< \theta_{i+1}\}$ be a sequence of singular fibers around $v$. That is, $H_{\theta_{i}}$ is a singular fiber which contains a singular leave which pass the vertex $v$. We denote the leaf in $H_{\theta}$ which passes $v$ by $\delta_{\theta}$. 

Take a sufficiently small number $\varepsilon >0$ so that there are no singularities in the interval $[\theta_{i}-\varepsilon,\theta_{i}+\varepsilon]$ except $H_{\theta_{i}}$.
We modify the closed braid $\widehat{\beta}=L$ and the surface $F$ so that they satisfy the following conditions.
\begin{enumerate}
\item $H_{\theta_{i} \pm \varepsilon} \backslash (H_{\theta_{i} \pm \varepsilon} \cap L)$ is the standard $n$-punctured disc $D_{n}$.
\item All vertices and a-arcs in a fiber $H_{\theta_{i} \pm \varepsilon}$ lie in the lower half of the disc $H_{\theta_{i} \pm \varepsilon} =D_{n}$.
\item The vertex $v$ lies at the leftmost position in the boundary of lower half of the disc $D_{n}$.
\item If all of $\{ \delta_{\theta} \}$ are b-arc in the interval $[\theta_{i}+\varepsilon,\theta_{i+1}-\varepsilon]$, then these b-arcs do not move in $[\theta_{i}+\varepsilon,\theta_{i+1}-\varepsilon]$.
\end{enumerate} 

 These conditions are achieved by the following way.
 First we isotope $\beta$ and $F$ so that the condition (3) and (4) holds. Now the condition (1) and (2) are achieved by the isotopy near singular fibers $H_{\theta_{i}}$, which preserves the condition (3) and (4).
 
 We denote this modified closed braid by $\widehat{\beta'}$. By cutting the closed braid $\widehat{\beta'}$ at the fiber $H_{\theta_{1}- \varepsilon}$, we obtain a braid $\beta'$.
 Since the above modification is an isotopy of the closed braid $\widehat{\beta}$ and the surface $F$ in the complement of the axis of $\widehat{\beta}$, so the braid $\beta'$ is conjugate to the original braid $\beta$. The benefit of this modification is the following. First, from the condition (1), we can decompose the braid $\beta'$ by the product of sub-braidings in each interval $[\theta_{i}-\varepsilon,\theta_{i}+\varepsilon]$ and $[\theta_{i}+\varepsilon,\theta_{i+1}-\varepsilon]$. The conditions (2)-(4) will be used to obtain the description of $\beta'$ in step 2. \\

\textbf{Step 2: Description of $\beta'$}\\

Our next task is to obtain an explicit description of the braid $\beta'$.
 First of all, we study the braiding near the singularities, that is, the braiding in the interval $[\theta_{i}-\varepsilon , \theta_{i}+\varepsilon]$. 
 
By seeing the moves of leaves on surfaces, we can obtain the explicit form of the braiding in each interval $[\theta_{i}-\varepsilon,\theta_{i}+\varepsilon]$. Especially, from the condition (2) and (3) in step 1, we have already known the first and the last configurations of leaves, so it is not hard to see how the leaves move. Here we simply state a result. We will present the moves of leaves by showing figures, which will convince the reader of the result. See \cite{bh} for detailed arguments to obtain a braid word via the braid foliation.

If the singularity in $H_{\theta_{i}}$ is an aa-singularity, then the braiding in $[\theta_{i}-\varepsilon,\theta_{i}+\varepsilon]$ is given by
\[ a_{1,j}^{\pm 1}=(\sigma_{1}\sigma_{2}\cdots\sigma_{j-1}\sigma_{j}^{\pm 1}\sigma_{j-1}^{-1}\cdots \sigma_{1}^{-1}). \]
 This is a band generator, which corresponds to a twisted band between two disc neighborhood of vertices.
 Thus, one aa-singularity produces one $\sigma_{1}^{\pm 1}$. See figure \ref{fig:aa-singularity}. 
 
 \begin{figure}[htbp]
 \begin{center}
\includegraphics[width=90mm]{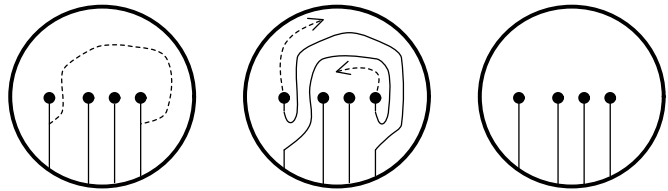}
 \end{center}
 \caption{Moves of of leaves near aa-singularity}
 \label{fig:aa-singularity}
\end{figure}
If the singularity in $H_{\theta_{i}}$ is an ab-singularity, then the braiding in $[\theta_{i}-\varepsilon,\theta_{i}+\varepsilon]$ is given by
\[(\sigma_{1}\sigma_{2}\cdots\sigma_{j}) \textrm{ or } (\sigma_{j}^{-1}\sigma_{j-1}^{-1}\cdots \sigma_{1}^{-1}). \]

We say an ab-singularity giving the former type of words {\it type a-b}. The name type a-b is derived from the fact that along the neighborhood of an ab-singularity, the leaf $\delta_{\theta}$ is changed from a-arc to b-arc. Similarly, an ab-singularity giving the latter type of words {\it type b-a}. See figure \ref{fig:ab-singularity}.

\begin{figure}[htbp]
 \begin{center}
\includegraphics[width=90mm]{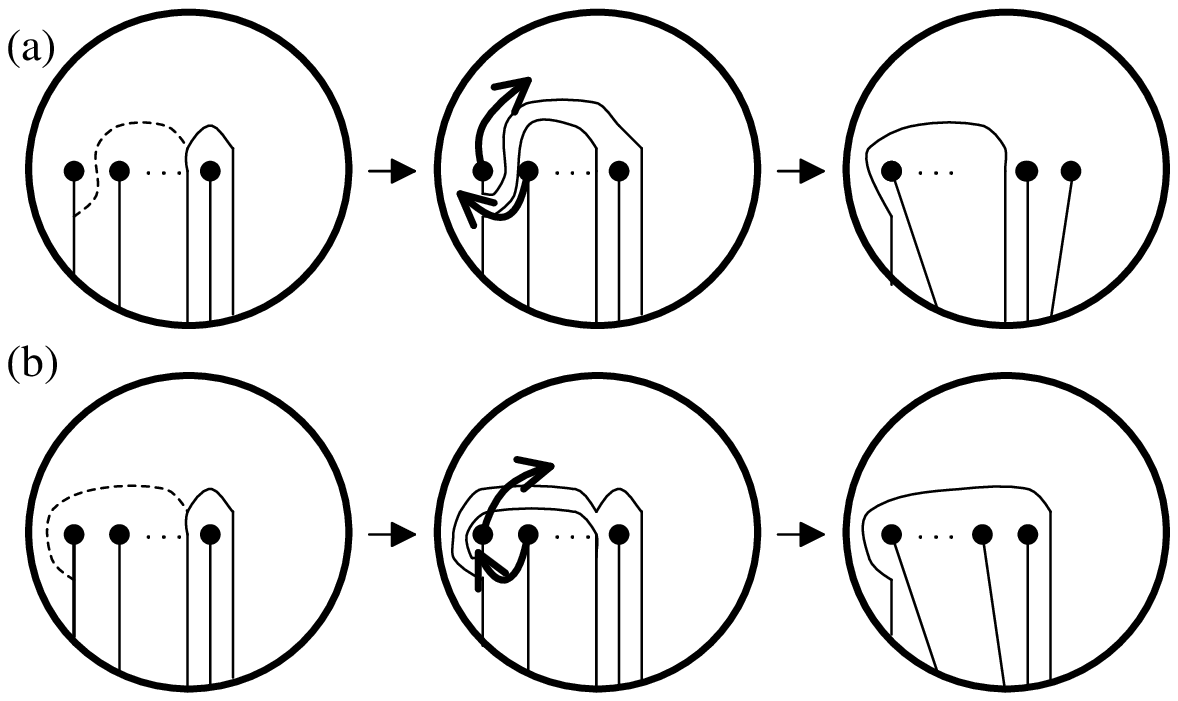}
 \end{center}
 \caption{Move of leaves near ab-singularity}
 \label{fig:ab-singularity}
\end{figure}

There are two kinds of ab-singularities. The first kind ab-singularity is an ab-singularity such that the a-arc is attached to the b-arc from the right side, as depicted in the figure \ref{fig:ab-singularity}(a). The second kind ab-singularity is an ab-singularity such that the a-arc is attached to the b-arc from the left side, as depicted in the figure \ref{fig:ab-singularity}(b). Although both kinds of ab-singularities provide the same braid words if their types are the same, but the configurations of b-arcs and braid strands after the ab-singularities are different. Therefore, in some cases we must distinguish them.

Around a bb-singularity, all a-arcs are not changed. Thus the link $L$, viewing as a boundary points of a-arcs, does not move in the neighborhood of the bb-singularity. Thus, a bb-singularity does not produce any $\sigma_{1}^{\pm 1}$. Therefore, when the singularity contained in $H_{\theta_{i}}$ is a bb-singularity, then the braiding in the interval $[\theta_{i}-\varepsilon,\theta_{i}+\varepsilon]$ is trivial.

Next we study a braiding outside the neighborhood of singularities, that is, in the interval $I=[\theta_{i}+\varepsilon,\theta_{i+1}-\varepsilon]$.
First observe there exist two types of such intervals $[\theta_{i}+\varepsilon,\theta_{i+1}-\varepsilon]$. Namely, 
\begin{enumerate}
\item An interval between two aa-singularities, or an interval between an aa-singularity and an ab-singularity.
\item An interval between an ab-singularity and an aa-singularity, or an interval between two bb-singularities.
\end{enumerate}

   We call each type of intervals {\it type A, type B} respectively because in a type A (resp. type B) interval, the leaf $\delta_{\theta}$ is always an a-arc (resp. b-arc).

Let $I$ be a type A interval and $\beta_{I}$ be the braiding in the interval $I$.
Let us denote the the $1$st strand of the braid $\beta_{I}$, which corresponds to the boundary point of the a-arc $\delta_{\theta}$, by $b_{1}$. Since in the interval $I$, the a-arcs $\delta_{\theta}$ do not form a singularity, the strand $b_{1}$ is not braided with other strands of $\beta_{I}$. Therefore the braid $\beta_{I}$ can be written as in the left diagram of figure \ref{fig:braiding_in_interval}.

Next we consider the braiding in a type B interval $J$. Since the b-arc $\delta_{\theta}$ is essential, the leaf $\delta_{\theta}$ separates each fiber $H_{\theta}$ into two components, both of which are pierced by $L$. Recall that we have modified the surface $F$ so that in the interval $J$, the leaf $\delta_{\theta}$ does not move ( condition (4) in step 1). Therefore the braiding in type B interval splits, so we can write the braiding as in the right diagram of figure \ref{fig:braiding_in_interval}. 

\begin{figure}[htbp]
 \begin{center}
\includegraphics[width=60mm]{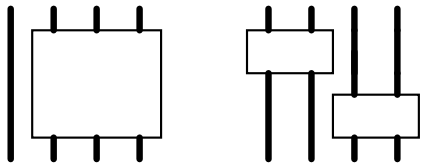}
 \end{center}
 \caption{Braiding in interval $[\theta_{i}+\varepsilon,\theta_{i+1}-\varepsilon]$}
 \label{fig:braiding_in_interval}
\end{figure}

\textbf{Step 3: Simplification procedure}\\

   We obtained an explicit description of the whole braiding of $\beta'$. The next step is to simplify the obtained braid so that it contains less $\sigma_{1}^{\pm1}$ as possible.
   We introduce two operations of simplification, the {\it ab-B simplification} and the {\it B-B simplification}. 

First we explain the ab-B simplification. Let $N$ be the $\varepsilon$-neighborhood of an ab-singular point and $I$ be the adjacent type $B$ interval. From step 2, the braiding in $N \cup I$ can be written as in the left diagram of figure \ref{fig:modification1}. 

If the ab-singularity is the first kind, that is, the a-arc is attached from right, then a braid box in $I$ which contains the strands $1$ can be shifted across the braiding in $N$ so that the modified braid contains only one $\sigma_{1}^{\pm1}$. 
If the ab-singularity is the second kind, that is, the a-arc is attached from left, then the braiding in $N$ is amalgamated into the adjacent braiding box in $I$ so that we can neglect the braiding derived from $N$. See figure \ref{fig:modification1}. In the figure, we show the type a-b singularity case. The type b-a singularity case is similar.  

\begin{figure}[htbp]
 \begin{center}
\includegraphics[width=100mm]{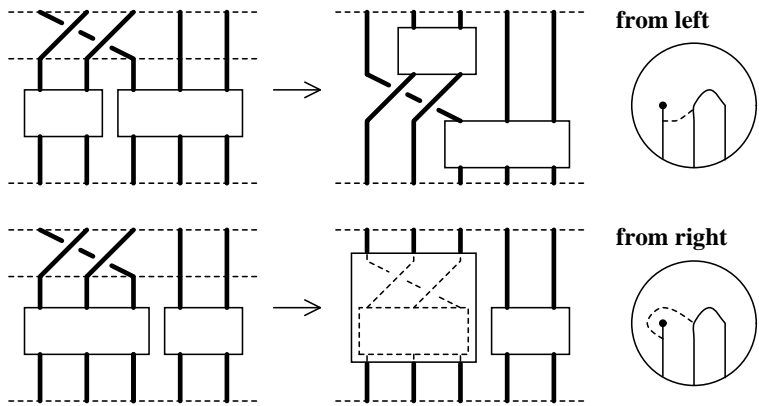}
 \end{center}
 \caption{ab-B simplifications} 
 \label{fig:modification1}
\end{figure}

   Next we explain the B-B simplification. Assume that two type B intervals $I_{1},I_{2}$ are adjacent to the interval $J$, which is a neighborhood of a bb-singularity. 
Then by exchanging the order of braid boxes in $I_{2}$, we can modify the braiding in the interval $I_{1}\cup J \cup I_{2}$ as in figure \ref{fig:modification2}. Then the obtained braiding in $I_{1} \cup J \cup I_{2}$ contains only one $\sigma_{1}$ and $\sigma_{1}^{-1}$.\\

\begin{figure}[htbp]
 \begin{center}
\includegraphics[width=80mm]{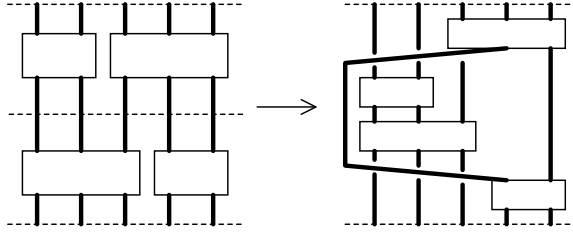}
 \end{center}
 \caption{B-B simplifications} 
 \label{fig:modification2}
\end{figure}

\textbf{Step 4: Proving lemma in $a=0$ or $b=0$ case}\\

Now we are ready to prove the lemma. 
First we consider the case $b=0$. In this case, there exist only aa-singularities around the vertex $v$.
Recall that we have shown that one aa-singularity provides only one $\sigma_{1}^{\pm 1}$. Thus, the braid $\beta'$ has at most $a$ $\sigma_{1}^{\pm1}$. Hence by proposition \ref{prop:dehornoyfloor}, we obtain $[\beta]_{D} < a$.
Since both $[\beta]_{D}$ and $a$ are integers, we conclude the inequality $[\beta]_{D} < a - \frac{1}{2}$ holds.

Next we consider the case $a=0$. In this case, there exist only bb-singularities around the vertex $v$.
Then, around the vertex $v$, there are $b$ type B intervals. By performing the B-B simplification which we described in step 3, we conclude that the braid $\beta'$ has at most $\frac{b}{2}$ $\sigma_{1}^{\pm1}$.
Thus, by proposition \ref{prop:dehornoyfloor}, we obtain the inequality $[\beta]_{D} < \frac{b}{2}$. 
If $b$ is even, then both $[\beta]_{D}$ and $\frac{b}{2}$ are integers. Thus we conclude that the inequality $[\beta]_{D} < \frac{b}{2} - \frac{1}{2}$ holds.
If $b$ is odd, then there must be adjacent bb-singular points with the same sign around the vertex $v$. Now the lemma \ref{lem:folchange} implies that we can decrease the valance of $v$ by one by moving the surface $F$. Therefore in this case we also obtain the inequality $[\beta]_{D} < \frac{b}{2}-\frac{1}{2}$.\\

\textbf{Step 5: Proving lemma in $a\neq 0$ and $b \neq 0$ case}\\

We complete the proof by showing the case that both $a$ and $b$ are non zero.

In this case, there must exist at least two ab-singularities around the vertex $v$. Then the cycle of singularities around the vertex $v$ are decomposed as the repetitions of the sub-cycles of the form 
\[ \{ \underbrace{aa \rightarrow aa \rightarrow \cdots \rightarrow aa}_{k\; \mbox{times}} \rightarrow ab \rightarrow \underbrace{ bb \rightarrow bb \rightarrow \cdots \rightarrow bb}_{l\;\mbox{times}} \rightarrow ab \} \] 

which contains $k$ aa-singularities and $l$ bb-singularities. We remark that $k$ and $l$ might be zero.

Recall that from step 2, one aa-singularity provide at most one $\sigma_{1}^{\pm1}$ and type A intervals and bb-singularities have no contributions to the number of $\sigma_{1}^{\pm 1}$.
Thus, the aa-singularities and type A intervals contribute the number of $\sigma_{1}^{\pm 1}$ by at most $k$.

We reduce the number of $\sigma_{1},\sigma_{1}^{-1}$ derived from ab-singularities and type B intervals using the ab-B and B-B simplifications described in step 3. First observe that there are three patterns of ab-singularities in the sub-cycle.
\begin{enumerate}
\item Both ab-singularities are from left.
\item One ab-singularity is from left, and the other is form right.
\item Both ab-singularities are from right.
\end{enumerate} 
  We consider each case separately.
 
 If both ab-singularities are from left, then the ab-B simplification amalgamates the braiding derived from the ab-singularities with the adjacent braid block derived from the type B interval. So we can neglect the braiding derived from the ab-singularities. Thus in this case, we only need to consider the contribution of $\sigma_{1}^{\pm1}$ derived from type B intervals. Then, the number of type B intervals is $l+1$, so these type B intervals contribute at most $\frac{l+1}{2}$ $\sigma_{1}^{\pm1}$. Consequently, in this case the sub-cycle contains at most $(\frac{l+1}{2}+k)$ $\sigma_{1}$ or $\sigma_{1}^{-1}$. 

 Next we assume that one of the ab-singularities is from left and the other is from right. Then by preforming the ab-B simplification, the ab-singularity from right is modified together with the adjacent type B interval so that they contribute one $\sigma_{1}$ or $\sigma_{-1}$. This ab-B simplification deletes one type B intervals. The braiding derived from the ab-singularity from left is amalgamated with the adjacent type B intervals, so we can neglect it. The number of remaining type B intervals is $l$, so the type B intervals contribute at most $\frac{l}{2}$ $\sigma_{1}^{\pm1}$. As a result, in this case there are at most $(\frac{l}{2}+1+k)$ $\sigma_{1}$ or $\sigma_{1}^{-1}$ in the sub-cycle.

 Finally, assume that the both of ab-singularities are from right. Then after ab-B simplifications, the ab-singularities, modified with the adjacent type B intervals, provide one $\sigma_{1}$ and $\sigma_{1}^{-1}$. The ab-B simplifications deletes two type B intervals. Thus the number of remaining type B intervals is $l-1$, and they contribute at most $(\frac{l-1}{2}$ $\sigma_{1}^{\pm1})$. Therefore in this case sub-cycle contains at most $(\frac{l-1}{2}+1+k)$ $\sigma_{1}^{\pm1}$.

Summarizing, we conclude that each sub-cycle contains at most $(k+1+\frac{l}{2})$ $\sigma_{1}^{\pm1}$.
Therefore if we want to the braid $\beta'$ contains $\sigma_{1}$ or $\sigma_{1}^{-1}$ as many as possible, then the number of sub-cycles in the cycle of singularities around the vertex $v$ must be one. 
Then there are $a-1$ aa-singularities and $b-1$ bb-singularities around $v$, so we conclude that the modified braid $\beta'$ has a word representative which has at most $a+\frac{b-1}{2}$ $\sigma_{1}$ or $\sigma_{1}^{-1}$. Since the original braid $\beta$ is conjugate to $\beta'$, by proposition \ref{prop:dehornoyfloor} we conclude that $[\beta]_{D} < a + \frac{b}{2}-\frac{1}{2}$.  
\end{proof}

   Now we are ready to prove theorem \ref{thm:main}.

\begin{proof}[Proof of theorem \ref{thm:main}]
   Let $L$ be an oriented link and $F$ be a Seifert surface of $L$ with maximal Euler characteristics. 
Take a closed braid representative $\widehat{\beta}$ of $L$ and isotope $F$ and $L$ so that $F$ satisfies the braid-foliation conditions.

If there exists a vertex of type (2,0), (1,1), (1,2), (0,2), (0,3) or (0,4), then lemma \ref{lem:estimate} shows $[\beta]_{D} < 2$, hence we obtain the inequality $[\beta]_{D}<\frac{3}{2}$. Therefore we can assume that there exist no vertices of such types. Thus, the Euler characteristic formula is now written as 
\[ -4 \chi(F)= V(2,1)+2V(3,0)+\sum_{v=4}^{\infty}\sum_{a=0}^{v}(v+a-4)V(a,v-a). \]

First assume that $F$ is foliated by only a-arcs. In this case, there exist exactly $n$ vertices on $F$ and exactly $-\chi(F) + n$ aa-singularities on $F$ because aa-singularity can be seen as a twisted band attached to disc neighborhoods of vertices. Therefore, the braid $\beta$ is conjugate to a braid written by a product of $n-\chi(F)$ band generators. Therefore from proposition \ref{prop:dehornoyfloor}, we establish the inequality $[\beta]_{D} < 1-\frac{\chi(F)}{n} \leq \frac{3}{2}- \frac{2\chi(F)}{n+2}$.

Now assume that $F$ contains both b-arc and a-arc. 
Let $(a,b)$ be a pair of integers such that the value $a+\frac{b}{2} = \frac{v+a}{2}$ is minimal among the all pair $(a,b)$ which satisfy $V(a,b) \neq 0$. Since $F$ contains b-arcs, there exist at least $n+2$ vertices in the foliation. Hence from the Euler characteristic formula, the inequality
\[ -4 \chi(F) \leq (2a+b-4)(n+2) \] 
holds. Therefore we obtain the inequality
\[ \frac{-2\chi(F)}{n+2} + 2 \leq a+\frac{b}{2}.\] 
Therefore lemma \ref{lem:estimate} gives desired estimation.
\end{proof}

We close this paper by giving a simple generalization.
There exist infinitely many families of left-invariant total orderings of the braid groups called {\it Thurston-type orderings}. The Dehornoy ordering is a special one of a Thurston type ordering \cite{sw}. Thurston type orderings have similar properties of the Dehornoy ordering. Especially the proposition \ref{prop:dehornoyfloor} also holds for the Thurston floor, which is defined by using a Thurston type ordering instead of the Dehornoy ordering. Hence our main theorem also holds if we use a Thurston floor instead of using the Dehornoy floor.


\begin{thebibliography}{[BF]}
\bibitem[BF]{bf} J.Birman, E.Finkelstein,
{\em{Studying surfaces via closed braids,}}
J. Knot theory Ramifications. , \textbf{7}, No.3 (1998), 267-334. 
\bibitem[BH]{bh} J.Birman, M.Hirsch,
{\em{A new algorithm for recognizing the unknot,}}
Geometry $\&$ Topology , \textbf{2}, (1998), 175-220. 
\bibitem[BM]{bm} J.Birman, W.Menasco,
{\em{Studying surfaces via closed braids VI: Non finiteness theorem,}}
Pacific J. of Math., \textbf{156}, No.2 (1992), 265-285. 
\bibitem[DDRW1]{ddrw1} P.Dehornor, I.Dynnkov, D.Rolfsen and B.Wiest, 
{\em{Why are the braids orderable ?,}}
Panoramas et Synth\'eses \textbf{14}, Soc. Math. France. 2002.
\bibitem[DDRW2]{ddrw2} P.Dehornoy, I.Dynnkov, D.Rolfsen and B.Wiest, 
{\em{Ordering Braids,}}
Mathematical Surveys and Monographs \textbf{148}, Amer. Math. Soc. 2008.
\bibitem[I]{i} T.Ito, 
{\em{Braid ordering and the geometry of closed braid,}}
e-print, arXiv:0805.1447v2  
\bibitem[MN]{mn} A.Malyutin,  N.Netsvetaev,
{\em{Dehornoy's ordering on the braid group and braid moves,}}
St.Peterburg Math. J. \textbf{15}, No.3 (2004), 437-448.
\bibitem[Ma]{ma} A.Malyutin,
{\em{Twist number of (closed) braids,}}
St.Peterburg Math. J. \textbf{16}, No.5 (2005), 791-813.
\bibitem[SW]{sw} H.Short, B.Wiest,
{\em{Ordering of mapping class groups after Thurston,}}
Enseign. Math. \textbf{46},(2000), 279-312.
\end{thebibliography}
\end{document}